\documentclass[a4paper,11pt]{amsart}
\usepackage[margin= 2.5 cm,bottom=19mm]{geometry}
\usepackage{thm-restate}
\usepackage{graphicx}
\usepackage{geometry}
\usepackage{amsthm}
\usepackage{amssymb}
\usepackage{amsmath}
\usepackage{comment}
\usepackage{hyperref}
\usepackage{xcolor}
\usepackage{enumerate}
\usepackage{enumitem}
\usepackage{listings}
\usepackage{complexity}
\usepackage{blkarray}
\usepackage{lmodern}
\usepackage[english]{babel}
\usepackage{accents}
\usepackage{float}
\usepackage{tikz}
\usetikzlibrary{decorations.pathmorphing}

\tikzset{snake it/.style={decorate, decoration=snake}}

\usetikzlibrary{calc}
\usetikzlibrary{decorations.pathreplacing}
\usetikzlibrary{positioning,patterns}
\usetikzlibrary{arrows,shapes,positioning}
\usetikzlibrary{decorations.markings}

\tikzstyle{edge}=[very thick]
\definecolor{bostonuniversityred}{rgb}{0.8, 0.0, 0.0}
\definecolor{arsenic}{rgb}{0.23, 0.27, 0.29}
\tikzstyle{diredge}=[postaction={decorate,decoration={markings,
		mark=at position .95 with {\arrow[scale = 1]{stealth};}}}]

\newcommand{\fitellipsis}[2] 
{\draw [fill=gray]let \p1=(#1), \p2=(#2), \n1={atan2(\y2-\y1,\x2-\x1)}, \n2={veclen(\y2-\y1,\x2-\x1)}
    in ($ (\p1)!0.5!(\p2) $) ellipse [ x radius=\n2/2+0cm, y radius=0.4cm, rotate=\n1];
}

\usepackage[font=small,labelfont=bf]{caption}
\usepackage[font=small,labelfont=normalfont,labelformat=simple]{subcaption}
\graphicspath{{figs/}}

\newtheorem{theorem}{Theorem}[section]
\newtheorem{lemma}[theorem]{Lemma}

\newtheorem{conjecture}[theorem]{Conjecture}

\theoremstyle{definition}

\author[Steiner]{Raphael Steiner}
\address[Steiner]{Department of Computer Science, Institute of Theoretical Computer Science, ETH Z\"{u}rich, Switzerland}
\email{\tt {raphaelmario.steiner}@inf.ethz.ch}

\thanks{The author was funded by an ETH Z\"{u}rich Postdoctoral Fellowship.
}

\date{\today}

\title{Improved bound for improper colorings of graphs with no odd clique minor}
\begin{document} 
\maketitle

\begin{abstract}
Strengthening Hadwiger's conjecture, Gerards and Seymour conjectured in 1995 that every graph with no odd $K_t$-minor is properly $(t-1)$-colorable, this is known as the \emph{Odd Hadwiger's conjecture}. 

We prove a relaxation of the above conjecture, namely we show that every graph with no odd $K_t$-minor admits a vertex $(2t-2)$-coloring such that all monochromatic components have size at most $\lceil \frac{1}{2}(t-2) \rceil$. The bound on the number of colors is optimal up to a factor of $2$, improves previous bounds for the same problem by Kawarabayashi (2008), Kang and Oum~(2019), Liu and Wood (2021), and strengthens a result by van den Heuvel and Wood (2018), who showed that the above conclusion holds under the more restrictive assumption that the graph is $K_t$-minor free. In addition, the bound on the component-size in our result is much smaller than those of previous results, in which the dependency on $t$ was non-explicit.

Our short proof combines the method by van den Heuvel and Wood for $K_t$-minor free graphs with some additional ideas, which make the extension to odd $K_t$-minor free graphs possible.
\end{abstract}

\section{Introduction}

\textbf{Preliminaries and Notation.} The terminology used in this paper is largely standard in graph theory, in the following we only outline some notions more specific to this paper. 

For an integer $k \ge 1$, we denote by $[k]:=\{1,\ldots,k\}$ the set of integers from $1$ to $k$. 
Given a graph $G$ and two vertex-disjoint subgraphs $H_1$ and $H_2$ of $G$, we say that $H_1$ and $H_2$ are \emph{adjacent} (in $G$) if there exist vertices $x \in V(H_1), y \in V(H_2)$ such that $xy \in E(G)$. By a \emph{component} of a graph we mean one of its connected components, and this is a subset of vertices (not a subgraph). 
Given a graph $G$, a vertex-coloring of $G$ is simply an assignment $c:V(G)\rightarrow S$ for some finite color-set $S$. It is called \emph{proper} if $c^{-1}(s)$ is an independent set in $G$ for every $s \in S$. For a (not necessarily proper) coloring $c$ of a graph $G$, a subset of vertices is called a \emph{monochromatic component}, if it is a component of one of the induced subgraphs $G[c^{-1}(s)]$ for some $s \in S$. For instance, a coloring of a graph is proper iff all its monochromatic components have size $1$.

Given an integer $t \ge 1$, a \emph{$K_t$-expansion} is a graph $F$ consisting of $t$ vertex-disjoint trees $(T_s)_{s=1}^t$, each two of them joined by exactly one additional edge. The $K_t$-expansion $F$ is said to be \emph{odd} if there exists a $2$-coloring $c$ of $V(F)$ such that the restriction of $c$ to any single tree $T_s$ forms a proper coloring of that tree, while every edge joining two distinct trees is monochromatic with respect to $c$, i.e., has the same colors at its endpoints.

Finally, we say that a graph $G$ \emph{contains $K_t$ as a minor}, or that it \emph{contains a $K_t$-minor}, if $G$ contains a subgraph which is a $K_t$-expansion. Similarly, $G$ is said to \emph{contain an odd $K_t$-minor} or that it contains \emph{$K_t$ as an odd minor} if $G$ has a subgraph which is an odd $K_t$-expansion. In the opposite cases, we say that $G$ is \emph{$K_t$-minor free} or \emph{odd $K_t$-minor free}, respectively. \medskip

The work in this paper is motivated by the famous graph coloring conjecture of Hadwiger.

\begin{conjecture}[Hadwiger's conjecture, 1943, \cite{hadwiger}]
For every integer $t \ge 2$, if $G$ is a graph not containing a $K_t$-minor, then $G$ is properly $(t-1)$-colorable. 
\end{conjecture}

A lot of work in graph theory has been inspired by and built around Hadwiger's conjecture, a survey of results and open problems covering the state of the art up until roughly $3$ years ago was written by Seymour~\cite{survey}. Hadwiger's conjecture has been proved for all values $t \le 6$ (see Robertson, Seymour and Thomas~\cite{robertson}), but remains open starting from $t=7$.
For a long time, the best asymptotic upper bound on the chromatic number of graphs with no $K_t$-minor has remained $O(t\sqrt{\log t})$ as established independently by Kostochka and Thomason~\cite{kostochka,thomason}. However, this bound was improved considerably recently, see~\cite{del,norine,norine2,postle,postle2}. The current best bound of $O(t\log\log t)$ was obtained half a year ago by Delcourt and Postle~\cite{del}.

Gerards and Seymour (see~\cite{gerards}, Section 6.5) proposed the following strengthening of Hadwiger's conjecture, called \emph{Odd Hadwiger's conjecture}. 

\begin{conjecture}[Odd Hadwiger's conjecture, 1995, \cite{gerards}]
For every integer $t \ge 2$, if $G$ is a graph not containing an odd $K_t$-minor, then $G$ is properly $(t-1)$-colorable. 
\end{conjecture}

To see that this conjecture indeed considerably strengthens Hadwiger's conjecture, consider for example $t=3$. While Hadwiger's conjecture in this case amounts to saying that forests (the $K_3$-minor free graphs) are $2$-colorable, the Odd Hadwiger's conjecture captures the more general satetement that all graphs without odd cycles (the odd $K_3$-minor free graphs) are $2$-colorable. In general, every $K_t$-minor free graph is also odd $K_t$-minor free, but there are odd $K_3$-minor free (i.e., bipartite) graphs which contain arbitrarily large clique minors.

The above conjecture has been verified for $t \le 4$ by Catlin~\cite{catlin}, and a solution for the case $t=5$ was announced by Guenin (cf.~\cite{survey}). For $t \ge 6$ the conjecture remains widely open. As for Hadwiger's conjecture, asymptotic upper bounds on the chromatic number of odd $K_t$-minor free graphs have been studied. An upper bound of $O(t\sqrt{\log t})$ was proved by Geelen et al. in~\cite{geelen} (see also~\cite{kawarabayashi2}), and recently this has been improved in~\cite{del,norine3,postle3,steiner}, with the current best bound being $O(t\log\log t)$ from~\cite{steiner}. For more results around the Odd Hadwiger's conjecture, we refer the interested reader to Chapter 7 of the survey~\cite{survey}.

The purpose of this paper is to prove the following relaxation of the Odd Hadwiger's conjecture, in which we allow our coloring to be improper, but instead require a constant bound (depending only on $t$) for the maximum size of monochromatic components. In return, our coloring uses much fewers colors than the known results for proper colorings.

\begin{theorem}\label{thm:main}
Let $t \ge 3$ be an integer. Then every graph $G$ without an odd $K_t$-minor admits a (not necessarily proper) vertex-coloring using $2t-2$ colors such that all monochromatic components have size at most $\lceil \frac{1}{2}(t-2) \rceil$.
\end{theorem}

Theorem~\ref{thm:main} lines up with a wide set of results on so-called \emph{improper colorings} of graphs with excluded minors. Instead of giving a long list of the individual results, let us just point to the comprehensive 70 page-survey on improper colorings written recently by Wood~\cite{woodsurvey} as well as to Chapter~6 of Seymour's survey~\cite{survey}. Two main variants of improper colorings have been studied: \emph{clustered} and \emph{defective} colorings. Given a graph $G$ and integers $k,c, d$, we say that a $k$-vertex coloring of $G$ has \emph{clustering} $c$ if all monochromatic components have size at most $c$, and we say that it has \emph{defect} $d$ if the maximum degree of all monochromatic components is bounded by $d$. Clearly, every $k$-coloring with clustering $c$ also has defect $c-1$. We may therefore rephrase Theorem~\ref{thm:main} by saying that for $t \ge 3$, every odd $K_t$-minor free graph is $2(t-1)$-colorable with clustering $\lceil \frac{1}{2}(t-2) \rceil$ and defect $\lceil \frac{1}{2}(t-4) \rceil$. The number of colors in our result improves upon previous results for this problem by Kawarabayashi~\cite{kawarabayashi}, Kang and Oum~\cite{kangoum} and Liu and Wood~\cite{liuwood}, summarized in Table~\ref{table} below. It is optimal up to a factor of $2$, as it was shown in~\cite{edwards,kangoum} that (odd) $K_t$-minor free graphs in general do not admit $(t-2)$-colorings with clustering bounded as a function of $t$.

As an additional advantage, our result also improves the dependency of the size of the clustering on $t$: Namely, the bounds on the clustering from~\cite{kangoum,kawarabayashi,liuwood} were only given as non-explicit functions of $t$ with a superlinear dependence on $t$.

\begin{table}[h]
\centering
\begin{tabular}{|c c c c|} 
 \hline
 & number of colors  & clustering & defect \\ [0.5ex] 
 \hline\hline
 Kawarabayashi~\cite{kawarabayashi} & $496t$ & $f_1(t)$ & $f_1(t)-1$ \\
 \hline
 Kang and Oum~\cite{kangoum} & $6t-9$ & $-$ & $f_2(t)$ \\
 \hline
 Kang and Oum~\cite{kangoum} & $10t-13$ & $f_3(t)$ & $f_3(t)-1$ \\
 \hline
 Liu and Wood~\cite{liuwood} & $8t-12$ & $f_4(t)$ & $f_4(t)-1$ \\ 
 \hline
 this paper & $2t-2$ & $\lceil \frac{1}{2}(t-2) \rceil$ & $\lceil \frac{1}{2}(t-4) \rceil$ \\ 
 [0.5ex] 
 
 \hline
\end{tabular}
\caption{Bounds for improper colorings of odd $K_t$-minor free graphs. The functions $f_1(t),\ldots,f_4(t)$ are used to indicate that the bound guaranteed on the defect or clustering is only dependent on $t$. The exact dependence of these functions on $t$ was however not made explicit in~\cite{kangoum,kawarabayashi,liuwood}.}\label{table}
\end{table}

Clustered and defective colorings of $K_t$-minor free graphs have also been extensively studied, see~\cite{edwards,heuvelwood,kawmoh,liuoum,liuwood,wood2}. Here the state of the art bounds are as follows: For defective coloring it was shown by Edwards et al.~\cite{edwards} that $K_t$-minor free graphs can be $(t-1)$-colored with bounded defects, and van den Heuvel and Wood~\cite{heuvelwood} proved that the defect can be bounded by $t-2$. For clustered coloring, it has been proved that $K_t$-minor free graphs can be $(t+1)$-colored with bounded clustering by Liu and Wood~\cite{liuwood}, and an optimal bound of $t-1$ colors was announced in 2017 by Dvořák and Norine~\cite{dvorine}. A weaker bound on the number of colors, however with an explicit bound on the clustering, was previously shown by van den Heuvel and Wood, namely that every $K_t$-minor free graph is $(2t-2)$-colorable with clustering $\lceil \frac{1}{2}(t-2) \rceil$. Theorem~\ref{thm:main} extends this result by van den Heuvel and Wood to odd $K_t$-minor free graphs. 

In the remainder of this paper, we give the proof of Theorem~\ref{thm:main}. Our proof follows closely a method introduced by van den Heuvel and Wood in~\cite{heuvelwood} to first establish a decomposition of the considered graphs into nicely structured disjoint subgraphs, from which a clustered coloring can then easily be obtained. Our decomposition result (Theorem~\ref{thm:decomposition}) is similar to a corresponding result for $K_t$-minor free graphs by van den Heuvel and Wood, but enhances it by some additional features, through which the extension from $K_t$-minor free graphs to odd $K_t$-minor free graphs becomes possible. 

\section{Proof of Theorem~\ref{thm:main}}

We need the following lemma proved by van den Heuvel and Wood in~\cite{heuvelwood}.

\begin{lemma}[cf.~Lemma 8, item (4) in~\cite{heuvelwood}]\label{lemma:connecting}
Let $G$ be a connected graph, and let $S \subseteq V(G)$ be such that $|S|=k \ge 1$. Let $H \subseteq G$ be an induced connected subgraph with a minimum number of vertices such that $S \subseteq V(H)$.

Then $H$ admits a partition of its vertex-set into two disjoint (possibly empty) subsets $A$ and~$B$ such that both $G[A]$ and $G[B]$ have all their connected components of size at most $\lceil \frac{k}{2} \rceil$.
\end{lemma}

The main idea of our proof is the following modified version of the above lemma, which will be useful for constructing odd minors.

\begin{lemma}\label{lemma:parityspanner}
Let $G$ be a connected graph, and let $S \subseteq V(G)$ be such that $|S|=k \ge 1$. Then there exists a connected induced subgraph $H \subseteq G$ with $S \subseteq V(H)$ such that the following hold:
\begin{enumerate}
    \item $H$ admits a partition of its vertex-set into two disjoint subsets $A$ and $B$ such that both $G[A]$ and $G[B]$ have maximum component-size at most $\lceil \frac{k}{2} \rceil$.
    \item The spanning bipartite subgraph of $H$ containing all the edges with one endpoint in $A$ and one endpoint in $B$ is connected.
    \item For every vertex $v \in V(G)\setminus V(H)$ which is connected in $G$ to at least one vertex in $V(H)$, there exist vertices $a \in A$ and $b \in B$ such that $av, bv \in E(G)$.
    
\end{enumerate}
\end{lemma}
\begin{proof}
By Lemma~\ref{lemma:connecting}, there exists at least one connected induced subgraph $H_0$ of $G$ such that $S \subseteq V(H_0)$ and a partition of $V(H_0)$ into subsets $A_0, B_0$ such that both $G[A_0], G[B_0]$ have maximum component-size at most $\lceil \frac{k}{2} \rceil$.

Now, let $(H,A,B)$ be a triple consisting of a connected induced subgraph $H \subseteq G$ with $S \subseteq V(H)$ and a partition $V(H)=A 
\cup B$ of its vertex-set such that $G[A]$ and $G[B]$ have maximum component-size at most $\lceil \frac{k}{2}\rceil$, chosen such that the number of edges in $H$ spanned between $A$ and $B$ is maximized among all possible choices of such triples.

We now claim that $H$ with the partition $A, B$ satisfies all three properties required by the lemma. Statement $(1)$ follows directly by our choice of the triple. To verify~$(2)$, suppose towards a contradiction that the spanning bipartite subgraph containing the edges in $H$ between $A$ and $B$ is disconnected. This would mean that there exists a partition of $V(H)$ into non-empty sets $X, Y$ such that there are no edges between $X \cap A$ and $Y \cap B$, and no edges between $X \cap B$ and $Y \cap A$ in $H$. Now define a new partition of $V(H)$ by $A':=(X \cap A) \cup (Y \cap B)$, and $B':=(X \cap B) \cup (Y \cap A)$. It is easy to see that since no edges in $G$ connect $X \cap A$ and $Y \cap B$ or $X \cap B$ and $Y \cap A$, every component of $G[A']$ or $G[B']$ is fully contained in either $A$ or $B$, and hence is contained in a component of either $G[A]$ or $G[B]$, and hence has size at most $\lceil\frac{k}{2}\rceil$. However, since $H$ is connected, there exists at least one edge $e \in E(H)$ with endpoints in $X$ and $Y$, which then must connect $X \cap A$ and $Y \cap A$ or $X \cap B$ and $Y \cap B$. In each case, $e$ is contained in the bipartite subgraph of $H$ spanned between $A'$ and $B'$. Also, every edge in $H$ in the bipartite subgraph spanned between $A$ and $B$ also has exactly one endpoint in $A'$ and in $B'$. Hence, $(H,A',B')$ is a triple satisfying all required properties which has strictly more edges between different sets in the partition than $(H,A,B)$. This is a contradiction to our choice of $(H,A,B)$, and proves $(2)$.

Finally, let us verify $(3)$. Towards a contradiction, suppose that there exists a vertex $v \in V(G)\setminus V(H)$ such that $v$ is connected in $G$ to at least one vertex in $V(H)$, but it does not have neighbors both in $A$ and in $B$. Then, w.l.o.g. (renaming $A$ and $B$ if necessary) we may assume that $v$ has no neighbors in $A$. Now, let $H':=G[V(H) \cup \{v\}]$ and put $A':=A \cup \{v\}$ and $B':=B$. Then $H'$ is a connected induced subgraph of $G$ with $S \subseteq V(H)\subseteq V(H')$. Since $v$ has no neighbors in $A$, every component in $G[A']$ or $G[B']$ is either a component of $G[A]$ or $G[B]$ and hence has size at most $\lceil\frac{k}{2}\rceil$, or is equal to $\{v\}$ and has size $1 \le \lceil\frac{k}{2}\rceil$. 

Furthermore, the number of edges in $H'$ spanned between $A'$ and $B'$ is strictly greater than the number of edges in $H$ spanned between $A$ and $B$, since in addition to these edges we have the edges incident to $v$ in $H'$. This again shows that $(H,A',B')$ is a triple satisfying all required properties with more edges between different sets in the partition than $(H,A,B)$, contradicting our maximality assumption. This shows that also $(3)$ is satisfied for $(H,A,B)$ and concludes the proof of the lemma.
\end{proof}

We next use the above lemma to prove the following decomposition result, which resembles a corresponding decomposition theorem proved by van den Heuvel and Wood in~\cite{heuvelwood} for $K_t$-minor graphs (compare Theorem~11 in~\cite{heuvelwood}). It extends part of the latter result with some additional features that will allow us to relate to odd minor containment instead of ordinary minor containment when building the decomposition of our graph. Once the decomposition theorem (Theorem~\ref{thm:decomposition} below) is established, Theorem~\ref{thm:main} will follow easily.

\begin{theorem}\label{thm:decomposition}
Let $t \ge 3$ be an integer, and let $G$ be a connected graph without an odd $K_t$-minor. Then there exists $\ell \in \mathbb{N}$ and a collection $H_1\ldots,H_\ell$ of vertex-disjoint induced connected subgraphs of $G$ with $V(H_1) \cup \cdots \cup V(H_\ell)=V(G)$ such that all of the following properties are satisfied for every $i \in [\ell]$:
\begin{enumerate}
    \item $H_i$ admits a partition of its vertex-set into two disjoint parts $A_i$ and $B_i$ such that in each of $G[A_i]$, $G[B_i]$, the maximum component-size is at most $\lceil \frac{t-2}{2} \rceil$.
    \item The spanning bipartite subgraph of $H_i$, containing all edges of $H_i$ with endpoints in $A_i$ and $B_i$, is connected.
    \item For every vertex $v \in V(G)\setminus (V(H_1) \cup \cdots \cup V(H_i))$ which is connected in $G$ to at least one vertex in $V(H_i)$, there exist vertices $a \in A_i, b \in B_i$ such that $av, bv \in E(G)$. 
    \item For every connected component $C$ of $G-(V(H_1) \cup \cdots \cup V(H_i))$, at most $(t-2)$ among the subgraphs $H_1, \ldots, H_i$ are adjacent to $G[C]$, and these subgraphs are pairwise adjacent to each other. 
\end{enumerate}
\end{theorem}
\begin{proof}
We construct the induced connected subgraphs $H_1,\ldots,H_\ell$ iteratively, maintaining the properties $(1)-(4)$ for all already constructed subgraphs in the sequence during the process.

Let $\mathcal{Z}$ denote the collection of all vertex-subsets $Z \subseteq V(G)$ such that $G[Z]$ is bipartite and connected (note that $\mathcal{Z} \neq \emptyset$, since every singleton-set in $V(G)$ belongs to $\mathcal{Z}$). Let now $X \in \mathcal{Z}$ be an inclusion-wise maximal member of $\mathcal{Z}$, and define $H_1:=G[X]$. By choice of $X$, the subgraph $H_1$ of $G$ is induced, bipartite and connected. Let us further verify that the invariants $(1)-(4)$ are satisfied: To verify $(1)$, we can simply let $A_1, B_1$ be the color classes of a bipartition of $H_1$. Item $(2)$ is satisfied trivially, since $H_1$ is connected and all edges of $H_1$ go between $A_1$ and $B_1$. For item $(3)$, consider any vertex $v \in V(G)\setminus X$ which has a neighbor in $X$. We have $X \cup \{v\} \notin \mathcal{Z}$ by our choice of $X$, and hence, $G[X \cup \{v\}]$ is non-bipartite. This means that $v$ must have neighbors both in $A_1$ and $B_1$, for otherwise either $(A_1 \cup \{v\},B_1)$ or $(A_1,B_1\cup \{v\})$ would form a bipartition of $G[X \cup \{v\}]$. Finally, this implies that there are neighbors $a \in A_1, b \in B_1$ of $v$, as required. Finally, item $(4)$ is trivially satisfied, since $t-2 \ge 1$. 

Next, suppose that for some integer $i \ge 1$ we have already constructed disjoint induced connected subgraphs $H_1,\ldots,H_i$ of $G$, each satisfying the invariants $(1)-(4)$, but such that $V(H_1) \cup \cdots \cup V(H_i) \neq V(G)$. Now, pick (arbitrarily) a connected component $C$ of the graph $G-(V(H_1) \cup \cdots \cup V(H_i))$. Let $Q_1,\ldots,Q_k$ be the (ordered) sublist of $H_1,\ldots,H_i$, containing exactly those subgraphs which are adjacent to $G[C]$. Since $G$ is connected, we have $k \ge 1$. By the invariant $(4)$ we furthermore know that $k \le t-2$ and that $Q_1, \ldots, Q_k$ are pairwise adjacent to each other. For every index $j \in [k]$, by definition there exists a vertex $v_j \in C$ such that $v_j$ has a neighbor in $Q_j$. Let $S:=\{v_1,\ldots,v_k\}$. Now, apply Lemma~\ref{lemma:parityspanner} to the connected graph $G[C]$ and the set $S$. We conclude that there exists an induced and connected subgraph $H$ of $G$ such that $S \subseteq V(H) \subseteq C$, equipped with a partition of its vertex-sets into subsets $A$ and $B$ such that 
\begin{itemize}
\item all components of $G[A]$ and $G[B]$ have size at most $\left\lceil \frac{|S|}{2} \right\rceil \le \lceil \frac{k}{2} \rceil \le \lceil \frac{t-2}{2} \rceil$,
\item the spanning bipartite subgraph of $H$ containing all edges of $H$ spanned between $A$ and $B$ is connected,
\item every vertex $v \in C\setminus V(H)$ which is connected to a vertex in $V(H)$ has neighbors both in $A$ and in $B$. 
\end{itemize}
We now finally define $H_{i+1}:=H$ and $A_{i+1}:=A, B_{i+1}:=B$, and claim that the extended sequence $H_1,\ldots, H_i, H_{i+1}$ still satisfies the invariants $(1)-(4)$. That the invariants $(1)$ and $(2)$ remain valid is an immediate consequence of the first two properties of $H$ listed above. 
Let us now verify that invariants $(3)$ and $(4)$ hold (and clearly, these need only be checked for the index $i+1$, since the claim is satisfied for smaller indices by assumption). 

For invariant $(3)$, let a vertex $v \in V(G) \setminus (V(H_1) \cup \cdots \cup V(H_i) \cup V(H_{i+1}))$ be given arbitrarily, and suppose that $v$ has at least one neighbor in $H_{i+1}$. Note that this implies that $v \in C$, since $C$ is a connected component of $G-(V(H_1) \cup \cdots \cup V(H_i))$ and $V(H_{i+1})=V(H) \subseteq C$. Therefore, by the third property of $H$ listed above, we conclude that $v$ has neighbors both in $A_{i+1}=A$ and in $B_{i+1}=B$. This verifies that the invariant $(3)$ remains satisfied. 

Finally, let us consider invariant $(4)$. For this purpose, let a connected component $C'$ of the graph $G-(V(H_1) \cup \cdots \cup V(H_i) \cup V(H_{i+1}))$ be given to us arbitrarily. Let $\mathcal{Q} \subseteq \{H_1,\ldots,H_i,H_{i+1}\}$ contain all the subgraphs adjacent to $G[C']$ in $G$. In order to verify invariant $(4)$ for $C'$, we need to show that $|\mathcal{Q}|\le t-2$ and that the members of $\mathcal{Q}$ are pairwise adjacent to each other.

Since $C$ is a connected component of $G-(V(H_1) \cup \cdots \cup V(H_i))$, we must either have $C' \cap C=\emptyset$ or $C' \subseteq C$, for otherwise $C \cup C' $ would induce a connected subgraph of $G-(V(H_1) \cup \cdots \cup V(H_i))$ and strictly contain $C$, a contradiction. For the same reason, if $C' \cap C = \emptyset$ then there is no edge in $G$ connecting $C$ to $C'$, and hence $C'$ in particular forms a connected component also of the graph $G-(V(H_1) \cup \cdots \cup V(H_i))$, and $H_{i+1} \notin \mathcal{Q}$. Therefore, in the case $C' \cap C = \emptyset$ the facts that $|\mathcal{Q}|\le t-2$ and that the members of $\mathcal{Q}$ are pairwise adjacent to each other follow from invariant $(4)$ for index $i$, which is satisfied by our initial assumptions. 

Moving on, suppose that $C' \subseteq C$. Then we clearly must have $\mathcal{Q} \subseteq \{Q_1,\ldots,Q_k,H_{i+1}\}$. Note that by invariant $(4)$ for index $i$ (applied with the component $C$), the subgraphs $Q_1,\ldots,Q_k$ are pairwise adjacent in $G$. Furthermore, since $S=\{v_1,\ldots,v_k\} \subseteq V(H)=V(H_{i+1})$ by our choice of $H$, we know that $H_{i+1}$ is adjacent to each of $Q_1,\ldots,Q_k$ in $G$. Hence, the members of $\mathcal{Q}$ are pairwise adjacent to each other. It remains to be shown that $|\mathcal{Q}| \le t-2$. Towards a contradiction, suppose that $|\mathcal{Q}| \ge t-1$. We have $k \le t-2$, and therefore this is only possible if $k=t-2$ and  $\mathcal{Q}=\{Q_1,\ldots,Q_{t-2},H_{i+1}\}$. 

We will now obtain the desired contradiction to the above assumption by constructing an odd $K_t$-expansion which is a subgraph of $G$ (clearly this does not exist by assumption on $G$). 
Let us denote by $i_1<i_2<\cdots<i_{t-1}=i+1$ the sequence of indices such that $\{Q_1,\ldots,Q_{t-2},H_{i+1}\}=\{H_{i_1},\ldots,H_{i_{t-1}}\}$. By invariant $(2)$ for $H_1, \ldots, H_i, H_{i+1}$, for every $s \in [i+1]$ we know that the bipartite spanning subgraph of $H_s$ containing the edges spanned between $A_s$ and $B_s$ is connected, and therefore admits a spanning tree $T_s$. This is a spanning tree of $H_s$ which uses only edges spanned between $A_s$ and $B_s$, for every $s \in [i+1]$. Furthermore, let $T$ be any fixed spanning tree of the connected graph $G[C']$. Finally, consider a $2$-color-assignment $c:\left(\bigcup_{j=1}^{t-1}{V(T_{i_j})}\right) \cup V(T) \rightarrow \{1,2\}$ to the vertices in the $t$ disjoint trees $T_{i_1},\ldots,T_{i_{t-1}}, T$ by piecing together proper $2$-colorings of the individual trees. To finish the construction of the odd $K_t$-minor, we need the following claim. 

\medskip
$(\ast)$ Any pair of two distinct trees from the collection $T_{i_1},\ldots,T_{i_{t-1}}, T$ is joined by at least one edge $xy \in E(G)$ satisfying $c(x)=c(y)$. 
\medskip

\begin{proof}[Proof of $(\ast)$]
Consider first the case that the pair of trees is of the form $T_{s_1}, T_{s_2}$ with $s_1<s_2$ and $s_1, s_2 \in \{i_1,\ldots,i_{t-1}\}$. Then, since $H_{s_1}, H_{s_2} \in \mathcal{Q}$ are adjacent, there exists a vertex $y \in V(T_{s_2})=V(H_{s_2})$ which is connected to a vertex in $V(T_{s_1})=V(H_{s_1})$. By invariant $(3)$, applied for the index $s_1$ and the vertex $y$, we find that $y$ must have neighbors $a \in A_{s_1}$ and $b \in B_{s_1}$ in $G$. Note that since $c$ restricted to $V(T_{s_1})$ is a proper coloring, we must have $c(a) \neq c(b)$. Hence, there exists $x \in \{a,b\}$ with $c(x)=c(y)$, and the edge $xy \in E(G)$ connecting $T_{s_1}$ and $T_{s_2}$ verifies $(\ast)$ in this case. 

Secondly, consider the case that the pair of trees is of the form $T_s, T$ for some $s \in \{i_1,\ldots,i_{t-1}\}$. Since $G[C']$ by definition is adjacent to every member of $\mathcal{Q}$, which includes $H_s$, analogous to the previous case there exists a vertex $y \in V(T)$ which is connected to $H_s$. Applying the invariant $(3)$ with the index $s$ and the vertex $y$ now yields that there are neighbors $a \in A_s, b \in B_s$ of $y$, and as above, we conclude that since $c(a)\neq c(b)$ there exists $x \in \{a,b\}$ with $c(x)=c(y)$. The edge $xy$ is monochromatic and connects $T_s$ and $T$, thus $(\ast)$ is verified also in the second case. 

This proves $(\ast)$.
\end{proof}
Now the collection of the $t$ disjoint trees $T_{i_1},\ldots,T_{i_{t-1}}, T$ in $G$, the coloring $c$ as well as the monochromatic edges guaranteed between each pair of trees by $(\ast)$ certify that $G$ contains an odd $K_t$-expansion. This is a contradiction to the assumption that $G$ is odd $K_t$-minor free, and hence, our above assumption that $|\mathcal{Q}| \ge t-1$ was wrong. This concludes the proof that also the invariant $(4)$ remains satisfied after extending the sequence $H_1,\ldots,H_i$ of subgraphs by $H_{i+1}$. 

Finally, since all the subgraphs $H_1, H_2, \ldots$ as defined above are non-empty, after finitely many steps the union of the subgraphs will cover all vertices of $G$, i.e., we will find an integer $\ell \ge 1$ such that $V(H_1) \cup \cdots \cup V(H_\ell)=V(G)$ forms a partition of $G$, with all four invariants $(1)-(4)$ satisfied for each index $i \in [\ell]$. This concludes the proof of the theorem.

\end{proof}

After having done the main bulk of work in the previous proof, we can now easily conclude Theorem~\ref{thm:main}.

\begin{proof}[Proof of Theorem~\ref{thm:main}]
Let $t \ge 3$ be an integer an let $G$ be any given odd $K_t$-minor free graph. W.l.o.g.~we may assume that $G$ is connected. We apply Theorem~\ref{thm:decomposition} to obtain a collection $H_1, \ldots, H_\ell$ of connected induced subgraphs of $G$ such that 

\begin{itemize}
    \item $V(H_1), \ldots, V(H_\ell)$ forms a partition of $V(G)$,
    \item every graph $H_i$ with $i \in [\ell]$ admits a $2$-coloring $f_i:V(H_i) \rightarrow \{1,2\}$ with monochromatic components of size at most $\lceil \frac{t-2}{2}\rceil$ (by property $(1)$ in Theorem~\ref{thm:decomposition}), and
    \item for every $1 \le i < \ell$ the subgraph $H_{i+1}$ is adjacent in $G$ to at most $t-2$ among the subgraphs $H_1,\ldots,H_{i}$ (by property $(4)$ in Theorem~\ref{thm:decomposition}, applied to the connected component of $G-(V(H_1) \cup \cdots \cup V(H_i))$ which contains $V(H_{i+1})$). 
\end{itemize}

Now define an auxiliary simple graph on the vertex-set $[\ell]$, in which two indices $i, j \in [\ell]$ are made adjacent if and only if the subgraphs $H_i$ and $H_j$ are adjacent in $G$. By the third item above, this graph is $(t-2)$-degenerate, and hence, it has chromatic number at most $(t-2)+1=t-1$. Fix a proper $(t-1)$-coloring $f:[\ell]\rightarrow [t-1]$ of this auxiliary graph. Now consider the product coloring $g:V(G) \rightarrow [t-1] \times \{1,2\}$, defined by $g(x):=(f(i), f_i(x))$ for every $x \in V(H_i)$. From the definition of the auxiliary graph and since $f$ is a proper coloring we have that every monochromatic component in $G$ with respect to the coloring $g$ must be fully included in $V(H_i)$ for some $i \in [\ell]$. But then it is a monochromatic component also of the coloring $f_i$ of $H_i$, and hence by the second item above cannot be of size more than $\lceil \frac{t-2}{2}\rceil$. Since $g$ uses a color-set of size $2(t-1)$, this proves the claim of the theorem.
\end{proof}

\end{document}